\documentclass[reqno,11pt]{amsart}
\usepackage[utf8]{inputenc}  
\usepackage{blkarray}
\usepackage{float} 

\newcommand{\ds}{\displaystyle}

\newcommand{\tensor}{\otimes}

\newcommand{\op}{\mathcal}

\newcommand{\cdc}{,\dots,}

\newcommand{\FT}{\mathsf{ft}}
\newcommand{\FTGK}{\mathsf{FT}}

\usepackage{fullpage}
\usepackage{amsmath}%
\usepackage{amsthm}
\usepackage{amsfonts}%
\usepackage{amssymb}%
\usepackage{graphicx}
\usepackage{xy,amsthm,enumerate,xypic,array}  
\usepackage{xcolor}

\usepackage{mathtools}

\DeclarePairedDelimiter\floor{\lfloor}{\rfloor}

\input{xy}
\xyoption{all}

\numberwithin{equation}{section}

\newtheorem{theorem}{Theorem}[section]
\theoremstyle{plain}

\newtheorem{corollary}[theorem]{Corollary}
\newtheorem{lemma}[theorem]{Lemma}
\newtheorem{proposition}[theorem]{Proposition}

\theoremstyle{definition}
\newtheorem{definition}[theorem]{Definition}

\allowdisplaybreaks[2]

\addtocounter{MaxMatrixCols}{2}

\begin{document}
\title{Lie graph homology model for $\mathfrak{grt}_1$.} 
\author{Benjamin C. Ward}

\maketitle
\begin{abstract}  This paper develops a new chain model for the commutative graph complex $\mathsf{GC}_2$ which takes Lie graph homology as an input.  Our main technical result is the identification of a large contractible complex of (certain) tadpoles and higher genus vertices of the Feynman transform of Lie graph homology.    Using this result we identify the anti-invarints of Lie graph homology in genus $2$ with relations between bracketings of conjectural generators of $\mathfrak{grt}_1$ in depth 2 modulo depth 3, unifying two {\it a priori} disparate appearances of the space of modular cusp forms in the study of graph homology. 
\end{abstract}

\section{Introduction}

The purpose of this article is to define and study a new chain model for commutative graph homology.  Let us begin by introducing the question which motivated this study.  It concerns a seemingly coincidental appearance of the space of modular cusp forms in both the commutative and Lie variants of graph homology.

The variant of commutative graph homology studied here-in was shown by Willwacher to be related to the Grothendieck-Teichm{\"u}ller Lie algebra $\mathfrak{grt}_1$ \cite{WTw}.  This Lie algebra contains the free Lie algebra generated by a known sequence of elements $\sigma_{2j+1}$, after Brown \cite{Brown}.  It also comes with a natural filtration, the depth, and Schneps gave in \cite{Schneps} (after Ihara \cite{Ihara}) a complete classification of relations satisfied by $\{\sigma_{2j+1},\sigma_{2i+1}\}$ in the associated graded $\op{F}_2(\mathfrak{grt}_1)/\op{F}_3(\mathfrak{grt}_1)$.  Namely, such relations are indexed by even period polynomials or equivalently modular cusp forms of weight $2i+2j+2$. 

The second instance of the space of cusp forms arose in Lie graph homology.  Conant, Kassabov, Hatcher and Vogtmann in \cite{CHKV} related Lie graph homology in genus $g$ with $n$ legs to a group extension, denoted $\Gamma_{g,n}$, of $Out(F_g)$ by $(F_g)^{\times n}$, where $F_g$ is the free group on $g$ generators.  Studying the Leray-Serre spectral sequence in the case $g=2$, they related the cohomology of $\Gamma_{2,n}$ to the group cohomology of $Out(F_2)\cong GL_2(\mathbb{Z})$ with coefficients in $H^\ast(F_2^{\times n})$.  In the case that $n$ is even their results show that the anti-invariants (invariants after tensoring the sign representation) are supported on $H^1(GL_2(\mathbb{Z}), H^n(F_2^{\times n}))$, hence live in a single degree, and that they can be identified with the space of modular cusp forms of weight $n+2$.

In this paper we give a new chain model for $\mathfrak{grt}_1$ built from Lie graph homology which realizes the above mentioned relations between brackets of conjectural generators, modulo higher depth, by precisely these anti-invariants of $H^\ast(\Gamma_{2,n})$.  More broadly, the novel feature of this model is the way in which features of commutative graph homology are reflected in Lie graph homology of the opposite parity and vice versa.  This model is manifest as a family of chain complexes $\mathsf{L}(g,n)$ spanned by certain colored and labeled graphs, which we presently describe.

\subsection{Description of $\mathsf{L}(g,n)$.}   We give an overview of the construction of the chain complex $\mathsf{L}(g,n)$ and refer to Section $\ref{modelsec}$ below for details.  Consider connected graphs with two types of vertices, which we color red and black.  Red vertices are labeled with a {\it positive} integer, which we call the genus.  Black vertices are unlabeled.  Such graphs are required to satisfy the following conditions:
\begin{itemize}
	\item Black vertices are stable (have at least three adjacent half edges).
	\item Graphs can have parallel edges only if at least one of the adjacent vertices is red.  
	\item Graphs can have tadpoles only at red vertices.
	\item The number of red vertices is not zero.
\end{itemize}
Our graphs can have legs (aka leaves), and a graph with $n$ legs whose total genus (first Betti number plus the sum of the red vertex labels) is $g$ is said to be of type $(g,n)$.  See Figure $\ref{fig:graph2}$.

\begin{figure}
	\centering
	\includegraphics[scale=.7]{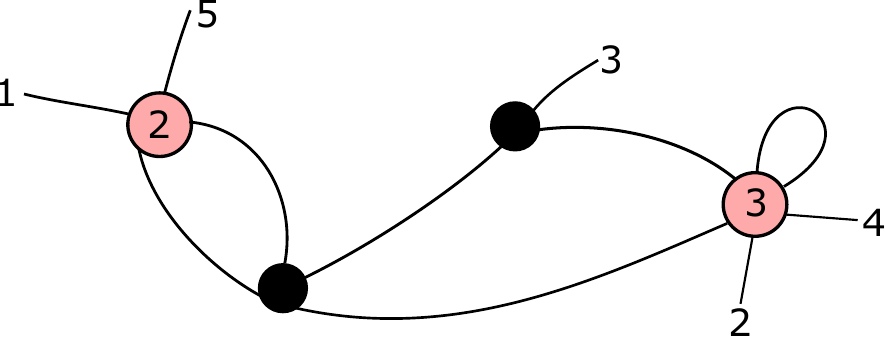}
	\caption{A graph of type $(8,5)$.}
	\label{fig:graph2}
\end{figure}

To each such graph $\gamma$ we associate a graded vector space by labeling the red vertices with the linear dual of the reduced homology of the appropriate $\Gamma_{g,n}$, i.e.\ we set
\begin{equation*}
\mathsf{L}(\gamma) =\bigotimes_{\substack{\text{red vertices} \\ \text{of } \gamma}} \widetilde{H}_\ast(\Gamma_{g(v),n(v)})^\ast\tensor \text{det}(E(\gamma)).
\end{equation*}
Here $\text{det}(E(\gamma))$ is the 1-dimensional top exterior power of the span of the edges of the graph, and has the effect of shifting degree and ordering the edges up to an even permutation.  

For $g\geq 1$ we define the graded vector space $\mathsf{L}(g,n)$ to be the colimit of the functor $\mathsf{L}(-)$ from the category of such graphs and their isomorphisms.  Picking a representative of each isomorphism class, we find the more concrete description:
\begin{equation*}
\mathsf{L}(g,n) = \bigoplus_{\substack{\gamma \text{ of} \\  \text{type } (g,n)}} \mathsf{L}(\gamma)_{Aut(\gamma)}.
\end{equation*}

It remains to describe the differential on each $\mathsf{L}(g,n)$.  Each complex carries a differential whose first term is induced by the classical Feynman transform differential (expanding one edge of the graph from a vertex of either color), but which includes higher order terms as well.  These higher order terms permit the vertex expansion of connected graphs with multiple edges, and occur only at the red vertices.  The expansion of a graph with $E$ edges must therefore index a higher homology operation on the family of graded vector spaces $H_\ast(\Gamma_{g,n})$ of degree $E-1$.  The existence of such homology operations is established by applying a version of the homotopy transfer theorem to the modular operad $H_\ast(\Gamma)$.  Although these operations depend on certain choices, much can be said about what must be true for any such choice, and this is the approach we take here-in.

\subsection{Summary of results}

Let $\mathsf{GC}^g_2$ be the commutative graph complex in genus $g$ with conventions as in \cite{WTw}.  In particular a graph in $\mathsf{GC}_2^g$ has no tadpoles, odd edges, and degree given by $E-2g$.  The differential is the sum over edge expansions.  Let $\mathsf{GC}^{g,n}_2$ be the analog of this graph complex whose graphs have $n$ ordered legs.

\begin{theorem}\label{introthm1}  Let $g\geq 1$.  There exists a short exact sequence of the form:
	\begin{equation}\label{ses1}
\Sigma^{2g}\mathsf{GC}^{g,n}_2 \hookrightarrow	\large{\substack{\text{acyclic} \\ \text{complex}}}\twoheadrightarrow\mathsf{L}(g,n).
	\end{equation}
\end{theorem}

The main technical achievement which makes this theorem possible is an extension of \cite[Theorem 1.1(2)]{CGP2} from the Feynman transform of the commutative modular operad to the Feynman transform of Lie graph homology, which identifies an acyclic subcomplex of (certain) loops and higher genus vertices, see Lemma $\ref{acyclicZ}$.  Here we view Lie graph homology together with its associated higher homology operations and use an $A_\infty$ analog of the Feynman transform, and so the resulting quotient complex, which appears in the middle of the short exact sequence $\ref{ses1}$, is the quotient of two acyclic complexes and so is itself acyclic.

The maps in this short exact sequence of Equation $\ref{ses1}$ are natural and do not depend on any choices, hence the ensuing homology isomorphisms are natural as well.  For example, when $n=0$, passing to the long exact sequence in homology and applying Willwacher's result from \cite{WTw} we see:

\begin{corollary}  The connecting homomorphism associated to Equation $\ref{ses1}$ induces an isomorphism
	\begin{equation*}
	\bigoplus_{g\geq 3}H^{2g-1}(\mathsf{L}(g,0)) \cong \mathfrak{grt}_1.
	\end{equation*}
\end{corollary}

Alternatively, when $n$ is not necessarily $0$, passing to the long exact sequence and applying results of Chan, Galatius and Payne \cite{CGP},\cite{CGP2} we see:
\begin{corollary}  Let $g>0$ and $i\geq 0$.  The connecting homomorphism associated to Equation $\ref{ses1}$ induces an isomorphism of $S_n$-modules
	\begin{equation*}
	H^{i}(\mathsf{L}(g,n)) \cong H^{i+1}(\Delta_{g,n}).
	\end{equation*}
\end{corollary}
Here $\Delta_{g,n}$ denotes the moduli space of tropical curves of unit volume (op.cit.).

Given these isomorphisms, it is natural to ask how known calculations and structures on the right hands sides are calculated or encoded on the left hand sides.  For example, the wheel graph in $(\mathsf{GC}^{2j+1}_2)^\ast$ represents a non-trivial class in homology.  A representative for this homology class in $\mathsf{L}(2j+1,0)^\ast$ can be read off by combining the results of \cite{WardWheels} with Lemma $\ref{acyclicZ}$ below.  Namely, it's supported on a bouquet of $2j-2$ circles at a genus $1$ vertex glued to a trivalent corolla of genus $0$, see Figure $\ref{s5}$. The genus $1$ vertex is labeled by a generator of the 1-dimensional space $H_{2j}(\Gamma_{1,2j+1})_{\Lambda_{2j+1}}$, where we write $V_{\Lambda_n}$ for the anti-invariants of an $S_n$ representation $V$, i.e.\ for $V\tensor_{S_n} \text{sgn}_n$.
 
 \begin{figure}
 	\centering
 	\includegraphics[scale=.44]{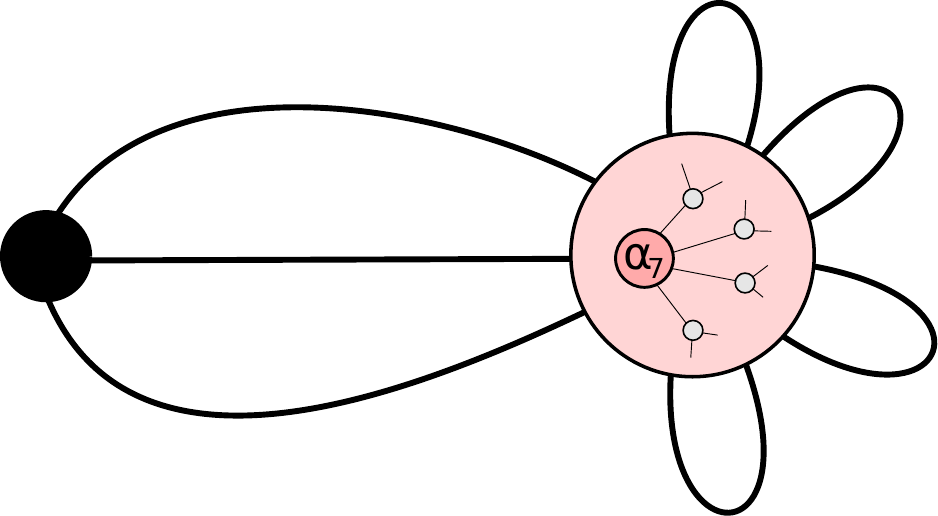}
 	\caption{A cycle in $\mathsf{L}(7,0)^\ast$ detecting the wheel graph is constructed using the generator $\alpha_7\in H_6(\Gamma_{1,7})$ composed with four copies of the commutative product in $H_0(\Gamma_{0,3})$.}
 	\label{s5}
 \end{figure}
 
In particular, the wheel graph can be detected by Lie graph homology classes in genus $1$.   It is therefore natural to ask if  the space of bracketings of conjectural generators lives in genus $2$, and indeed we prove:

\begin{theorem}\label{2thm}  There is an isomorphism
\begin{equation}\label{iso2}
\bigoplus_{\text{even } n\geq 4} H^{\ast}(\mathsf{L}(2,n)_{\Lambda_n}) \cong \op{F}_2(
\mathfrak{grt}_1)/\op{F}_3(
\mathfrak{grt}_1)
\end{equation}
induced by 
\begin{equation}\label{mapp}
\includegraphics[scale=.5]{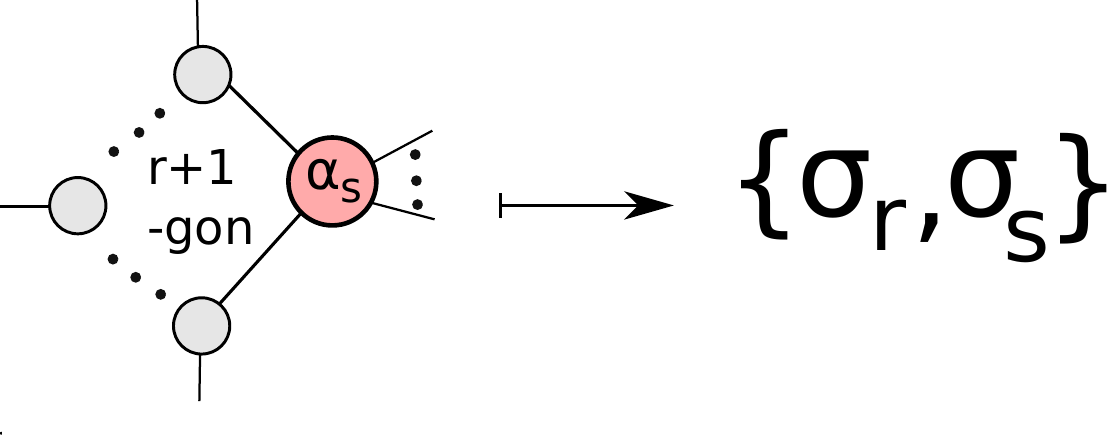}\\
\end{equation}
\end{theorem}

See Theorem $\ref{g2thm}$ below for a precise fixing of conventions.  
The heart of the proof of this Theorem is the computation of a basis for the homology of the subcomplex of non-corollas of $\mathsf{L}(2,n)_{\Lambda_n}$.  We then compute the image of each homology class under the connecting isomorphism and identify the result (up to a straight-forward parity shift) with the image of the isomorphism of M. Felder in \cite{MFelder}.  This result shows that any coherent choice of higher operations determines an isomorphism between $H_\ast(\Gamma_{2,n})^\ast$ and the linear relations among the brackets in $(\op{F}_2(\mathfrak{grt}_1)/\op{F}_3(\mathfrak{grt}_1))_{n+2}$. 

One consequence of our computations is that for $n\geq 4$ and even:
	\begin{equation}
	dim(H^{i+2}(\Delta_{2,n})_{\Lambda_n}) + dim(H^i(\Gamma_{2,n})_{\Lambda_n}) = \begin{cases}
	\floor{\frac{n-2}{4}} & \text{ if } i=n+1 \\ 
	0 & \text{else.}
	\end{cases}
	\end{equation}
I.e.\  $H^{\ast}(\Delta_{2,n})_{\Lambda_n}$, computed in \cite[Theorem 6.2]{CCTW} and $H^\ast(\Gamma_{2,n})_{\Lambda_n}$, computed in \cite[Theorem 2.10]{CHKV} are formal consequences of each other. It would be desirable to extend this relationship beyond just the anti-invariants.

Theorem $\ref{2thm}$ relates genus $2$ of Lie graph homology to depth $2$ of $\mathfrak{grt}_1$, and one may ask if some such relationship continues beyond $2$.  One way to address this question would be to probe the leg-free graph complex by gluing in a (black) $n$-corolla to the $n$ anti-invariant legs of the graphs supporting elements in $\mathsf{L}(g,n)_{\Lambda_n}$.  Such an assignment would not be dg on the nose, however for $n \geq g$ it is possible to filter the result by the largest valence vertex and hence to get a dg map to the associated graded.  The associated spectral sequence comparison between these two objects may be deserving of future study.

\tableofcontents

\section{Background}

In this section we give a brief overview of the necessary background material.  For more detail we refer to \cite{GeK2} for modular operads, \cite{WardMP} for weak modular operads, and \cite{KontSymp},\cite{WTw} and \cite{CHKV} and the references there-in for graph homology.

Throughout, we work over a field $k$ of characteristic zero and use cohomological grading conventions, except for when expressly working with homology.  Note that \cite{WardMP} used homological conventions.

\subsection{Modular graphs and modular operads}  

In this paper, a graph $\gamma$ refers to a nonempty set of vertices $V$ and a set of half-edges $H$, along with an adjacency map $a\colon H \to V$, which specifies the vertex to which a half edge is adjacent and an involution $H\to H$.  The orbits of the involution of size two are called edges, and the fixed points are called legs.  A modular graph is a connected graph along with a bijection $\{1\cdc n\}\cong \text{legs}(\gamma)$ and a genus labeling of the vertices $g\colon V\to \mathbb{Z}_{\geq 0}$, such that the stability condition $2g(v)+a^{-1}(v) \geq 3$ is satisfied $\forall \ v\in V$.

A (dg) modular operad \cite{GeK2} is, roughly, a collection of chain complexes $\op{M}(g,n)$ along with an equivariant composition map corresponding to each modular graph.  In particular if $\gamma$ is a modular graph with total genus $g:= \beta_1(\gamma) + \sum_{V(\gamma)} g(v)$ and $n$ legs, there is a contraction map 
$$\op{M}(\gamma):=\tensor_{V(\gamma)} \op{M}(g(v), a^{-1}(v))\stackrel{\mu_\gamma}\to \op{M}(g,n)$$
 which is invariant with respect to the action of leg fixing automorphisms of the graph and equivariant with respect to the symmetric group action on the leg labels.  A necessary variant of modular operads defines the composition maps corresponding to graphs whose edges are given degree $1$.  These are called $\mathfrak{K}$-modular operads.  More specifically if we define $\mathfrak{K}(\gamma)$ to be the top exterior power of the set $E(\gamma)$ of edges of the graph $\gamma$, concentrated in degree $|E(\gamma)|$, then a $\mathfrak{K}$-modular operad has equivariant composition maps $\op{M}(\gamma)\tensor \mathfrak{K}(\gamma)\to \op{M}(g,n)$.

Given a dg modular operad $A$, its Feynman transform is the $\mathfrak{K}$-modular operad defined by chain complexes:
\begin{equation}\label{FT}
	\FTGK(A)(g,n):= \bigoplus A(\gamma)^\ast\tensor_{Aut(\gamma)} \mathfrak{K}(\gamma),
\end{equation}
with direct sum taken over isomorphism classes of modular graphs of total genus $g$ and with $n$ legs. The $\mathfrak{K}$-twisted modular operadic compositions are free, given by grafting legs.  The differential is given by the linear dual of the sum of the one-edged contraction maps.  This differential squares to zero because the contraction of two edges in either order produces terms which add to $0$ in the $\mathfrak{K}$ factor. Similarly, given a $\mathfrak{K}$-modular operad $\op{B}$, its Feynman transform is the free untwisted modular operad.  It has an analogous differential, which in this case is square zero because the original contractions were $\mathfrak{K}$-twisted.

By an ``homogeneous element'' in the Feynman transform we refer to a pure tensor supported on a single summand.  For a summand corresponding to a graph $\gamma$, we refer to $\gamma$ as the graph underlying the elements. 

\subsection{Recollection of Lie graph homology}
Recall that a cyclic operad is equivalent to a modular operad whose higher genus terms are zero. One such example of a cyclic operad is the Lie operad $\mathsf{Lie}$ whose underlying operad encodes Lie algebras.  Any cyclic operad also determines a $\mathfrak{K}$-twisted modular operad by taking its operadic suspension, along with an additional downward shift in degree, and then extending by $0$ in higher genus.  We thus have the $\mathfrak{K}$-twisted modular operad $\Sigma^{-1} s \mathsf{Lie}$.  In particular $(\Sigma^{-1} s \mathsf{Lie})(g,n)$ is zero unless $g=0$, in which case it is concentrated in degree $n-3$.  Observe that $n$ here denotes the total number of legs not just the inputs.

By Lie graph homology we mean the modular operad $H_\ast(\FTGK(\Sigma^{-1} s \mathsf{Lie}))$.  In \cite{CHKV}, Conant, Hatcher, Kassabov and Vogtmann study a family of groups $\Gamma_{g,n}$ whose homologies form a modular operad such that $H_\ast(\Gamma)\cong H_\ast(\FTGK(\Sigma^{-1} s \mathsf{Lie}))$, where $H_\ast(\Gamma)(g,n):=H_\ast(\Gamma_{g,n})$.  We will use the following subset of their results:

\begin{lemma}\cite[Proposition 2.7]{CHKV}\label{Lieg1}  The $S_n$ representation $H_i(\Gamma_{1,n})$ is zero, unless $n-1 \geq i \geq 0$ and $i$ is even, in which case $H_i(\Gamma_{1,n}) \cong V_{n-i,1^i}$. 
\end{lemma}

\begin{lemma}\cite[Theorem 2.10]{CHKV}  Let $n$ be even.  The number of copies of the alternating representation appearing in the $S_n$-module $H_i(\Gamma_{2,n})$ is $0$ unless $i=n+1$ in which case it is $\floor{\frac{n-2}{4}} - \floor{\frac{n}{6}}$.
\end{lemma}

For ease of notation we denote Lie graph homology by $H_\ast(\Gamma)$ from now on.

\subsection{The weak Feynman transform}
Let us recall the notion of weak modular operads (aka $A_\infty$-modular operads) as introduced in \cite{WardMP}.  A modular operad has an operation for every modular graph, but each such operation can be decomposed into a sequence of one-edged compositions.  The one-edged compositions may be thought of as generating operations.  A weak modular operad, on the other hand, has a {\it generating operation} for every modular graph.  These generating operations may themselves be composed to form operations corresponding to every nested graph.
 
To be precise, the generating operation in a weak modular operad corresponding to a modular graph $\gamma$ of type $(g,n)$ is a degree $-1$ linear map
\begin{equation}
	\mathfrak{K}(\gamma)\tensor A(\gamma) \stackrel{\tilde{\mu}_\gamma}\longrightarrow A(g,n)
\end{equation}
Fixing an order of the set of edges $E=E(\gamma)$, this operation induces a degree $|E|-1$ map
\begin{equation}\label{mp}
A(\gamma) \stackrel{\mu_\gamma}\longrightarrow A(g,n).
\end{equation}
The operation $\mu_\gamma$ depends on the chosen order of the edges, but only up to sign.  Note that if $|E|=1$ there is a unique possible order and we recover the form of the generating operations in a modular operad.  Modular operads are weak modular operads whose higher operations, meaning $\mu_\gamma$ for which $|E(\gamma)|>1$, vanish.  Below, following {\it op.\ cit.}, these higher operations will be referred to as Massey products.

The Feynman transform may be generalized to take weak modular operads as input and produce $\mathfrak{K}$-twisted modular operads as output.  We denote this construction by lowercase $\FT$ to contrast it with the classical Feynman transform $\mathsf{FT}$.  The underlying $\mathfrak{K}$-twisted modular operad is the same
\begin{equation}
	\FT(A)(g,n):= \bigoplus A(\gamma)^\ast\tensor_{Aut(\gamma)} \mathfrak{K}(\gamma),
\end{equation}
but the differential is the linear dual of the sum of all generating operations.  If we view a modular operad as a weak modular operad by extension by zero, then its Feynman transform and weak Feynman transform coincide.

Composition of operations in a weak modular operad obey differential conditions analogous to the differential conditions for operations in an $A_\infty$ algebra.  Given a weak modular operad whose internal differential is $0$, these differential conditions collapse, and we may forget the higher operations to produce a modular operad.  Conversely we say a weak modular operad structure extends a modular operad structure if forgetting the higher operations of the former produces the latter.  Unlike extension by zero, the act of forgetting operations does not commute with the (weak) Feynman transform.

The advantage of the weak Feynman transform is that it satisfies the following homotopy transfer theorem:

\begin{theorem}\label{mpthm}\cite{WardMP}  Let $A$ be a dg modular operad and consider $H_\ast(A)$ with its induced modular operad structure.  There exists an extension of $H_\ast(A)$ to a weak modular operad for which $A\sim H_\ast(A)$ as weak modular operads.
\end{theorem}
The symbol $\sim$ in the statement of the theorem denotes the existence of an $\infty$-quasi isomorphism.  The important consequence of such a relation is that it induces an honest quasi-isomorphism (aka a homology isomorphism) between $\FTGK(A)$ and $\FT(H_\ast(A))$ upon taking the weak Feynman transform.

\begin{corollary}\label{mpcor}  Let $B$ be a $\mathfrak{K}$-twisted modular operad and define $A= H_\ast(\FTGK(B))$.  The modular operad structure on $A$ may be extended to a weak modular operad structure such that $H_\ast(\FT(A))\cong H_\ast(B)$ as $\mathfrak{K}$-twisted modular operads.
\end{corollary}

Let us apply Corollary $\ref{mpcor}$ in the case $A=H_\ast(\Gamma)$ and $B=H_\ast(B)= \Sigma^{-1} s\mathsf{Lie}$ to extend the modular operad structure on $H_\ast(\Gamma)$ to a weak modular operad.  We choose such an extension and denote this weak modular operad by $(h_\ast(\Gamma),\mu)$.   Since the $\mathfrak{K}$-twisted modular operad $\Sigma^{-1} s\mathsf{Lie}$ is $0$ in genus $g>0$, Corollary $\ref{mpcor}$ tells us:
\begin{corollary}\label{aacyclic} $\FT(h_\ast(\Gamma))(g,n)$ has no homology for $g>0$.
\end{corollary}

The extension in Theorem $\ref{mpthm}$ is not unique in the naive sense; it is unique in a suitably derived sense.  This is reflected by the fact that Corollary $\ref{aacyclic}$ is true for every such extension.

\section{Lie graph homology model}\label{modelsec}  

The goal of this section is to prove Theorem $\ref{introthm1}$.  Our starting point is a coherent choice of Massey products on the modular operad $H_\ast(\Gamma)$, after Theorem $\ref{mpthm}$, and we denote the resulting weak modular operad by $(h_\ast(\Gamma),\mu)$. As above, the choice of extension is not unique and in this section we consider what must be true for any such extension.  In what follows we will view $\FT(h_\ast(\Gamma))(g,n)$ as having cohomological grading conventions, and so we do not negate the degree of a homology class when taking its linear dual.

\subsection{Degree 0 and Genus 0}
\begin{lemma}
If $\gamma$ is a graph of total genus 0 then $\mu_\gamma=0$ unless $\gamma$ has exactly one edge.
\end{lemma}
\begin{proof}
	Koszul duality between the commutative and Lie operads implies that for each $n\geq 3$,  $H_\ast(\Gamma_{0,n})= H_0(\Gamma_{0,n})\cong \mathsf{Com}(0,n)\cong k$.  Given a graph $\gamma$ with $E$ edges, $\mu_\gamma$ has degree $E-1$ and so must vanish unless $E-1=0$.
\end{proof}
Above we have written $\mathsf{Com}$ for the commutative (cyclic) operad.  By abuse of notation we write $\mathsf{Com}$ and $\overline{\mathsf{Com}}$ for the modular operads associated to it by extension by $0$ and the modular envelope respectively.  Explicitly, for a stable pair $(g,n)$:
\begin{equation*}
\mathsf{Com}(g,n) = \begin{cases} 
k & \text{ if } g=0 \\
0 & \text{ if } g\geq 1
 \end{cases}  \ \ \  \text{ and }
\ \ \ 
 \overline{\mathsf{Com}}(g,n) = \begin{cases} 
 k & \text{ if } g=0 \\
 k & \text{ if } g\geq 1.
 \end{cases} 
\end{equation*}
We may view these modular operads as weak modular operads with trivial higher operations.  There are arity-wise surjective morphisms of weak modular operads 
\begin{equation*}
h_\ast(\Gamma) \to \overline{\mathsf{Com}} \to \mathsf{Com}.
\end{equation*}
The fact that the former is a weak modular operad map follows from the fact that any higher operation must land in positive degree.  Taking the weak Feynman transform, we have a sequence of arity-wise injections of $\mathfrak{K}$-twisted modular operads:
\begin{equation}\label{inj}
\FT(\mathsf{Com}) \hookrightarrow \FT(\overline{\mathsf{Com}}) \hookrightarrow \FT(h_\ast(\Gamma)).
\end{equation}
This latter map is the inclusion of those graphs labeled by degree $0$ classes at each vertex.

\subsection{Bigrading}

For each stable $(g,n)$, we write $\FT(h_\ast(\Gamma))(g,n)^{r,s}$ for the span of homogeneous elements having $r$ edges and internal degree labels summing to $s\geq 0$.  In particular both $r$ and $s$ are non-negative and the total degree is $r+s$. The differential expands $m\geq 1$ edges:
\begin{equation}\label{d}
\FT(h_\ast(\Gamma))(g,n)^{r,s}\stackrel{\partial}\to \bigoplus_{m\geq 1} \FT(h_\ast(\Gamma))(g,n)^{r+m,s-m+1}
\end{equation}

The differential applied to $\FT(h_\ast(\Gamma))(g,n)^{r,0}$ can only expand one edge (lest $s<0$) and so $\FT(h_\ast(\Gamma))(g,n)^{\ast,0}\subset \FT(h_\ast(\Gamma))(g,n)$ is a subcomplex.  Since $h_0(\Gamma)\cong \overline{\mathsf{Com}}$ as modular operads, we may identify $\FT(h_\ast(\Gamma))(g,n)^{\ast,0}$ with the subcomplex $\FT(\overline{\mathsf{Com}})$ in Equation $\ref{inj}$.

\begin{lemma}\label{upper}  Given $(g,n)$ we have:
		\begin{equation*}	
	\FT(h_\ast(\Gamma))(g,n)^{r,s} = \begin{cases}
	0 & \text{ if } s> 2g+n-3 \\ 
	0 & \text{ if } s= 2g+n-3 \text{ and } r \neq 0 \\
	H_{2g+n-3}(\Gamma_{g,n})^\ast & \text{ if } s= 2g+n-3 \text{ and } r = 0  \\ 
	0 & \text{ if } r+s > 3g+n-3  \\
	0 & \text{ if } r+s \geq 3g+n-3 \text{ and } s \neq 0. \\
	\end{cases}
	\end{equation*}
\end{lemma}

\begin{proof}  
This lemma can be interpreted as stating that $\FT(h_\ast(\Gamma))(g,n)$ is supported on the shaded area in the following cartoon picture:
\begin{center}
		\includegraphics[width=0.7\linewidth]{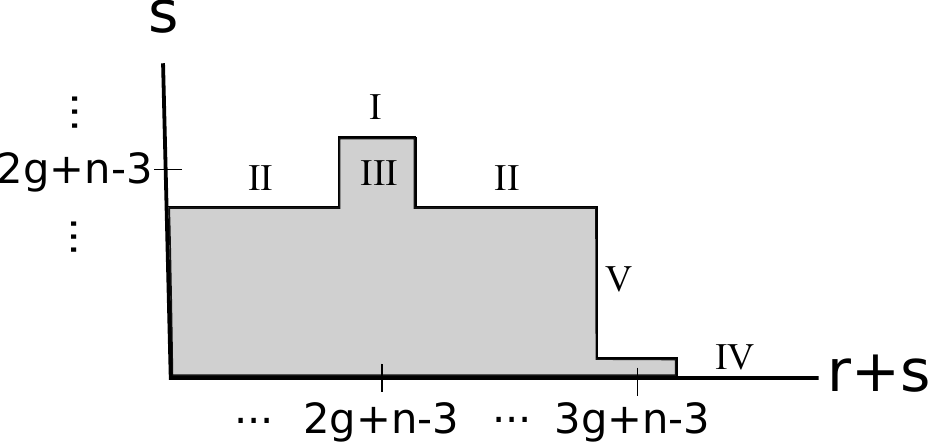}
\end{center}
The Roman numerals indicate the five statements inside the bracket in the statement of the lemma.

We will prove the statements by contraposition.  Suppose that $\FT(h_\ast(\Gamma))(g,n)^{r,s} \neq 0$ for some $(r,s)$.  Choose a non-zero homogeneous element supported on a graph $\gamma$ with vertices $v_1\cdc v_m$.  Each vertex carries a label by a class in $H_\ast(\Gamma_{g_i,n_i})^\ast$ and the degree of this class is at most $2g_i+n_i-3$, by \cite[Proposition 2.6]{CHKV}.  Therefore
	\begin{equation}\label{int}
	s\leq \sum_{i=1}^m 2g_i+n_i-3
	\end{equation}
The total genus of the element satisfies  $g=\beta_1(\gamma)+\sum g_i$ and counting flags gives us $\sum n_i = 2r+n$.  Therefore Equation $\ref{int}$ implies
\begin{equation*}
s\leq (2g-2\beta)+(2r+n)-3m = 2g+n+2(r-\beta)-3m 
\end{equation*}
Finally, from the Euler characteristic we know $m-r = 1-\beta$, hence
\begin{equation}\label{five}
s\leq  2g+n+2m-2-3m=2g+n-2-m
\end{equation}
The first statement then follows from the fact that $m\geq 1$.  

For the second statement, assuming $r\neq 0$  there are two cases.  First if $\gamma$ has more than one vertex, we see the statement follows as above.  So now suppose it has a lone vertex.  Since $r\neq 0$, it must have a loop.  Thus $g_1<g$ and so the above inequality must be strict.

The third statement follows immediately from the definition of the Feynman transform.

For the fourth statement we add $r$ to both sides of Equation $\ref{five}$ and use the Euler characteristic and the fact that $\beta_1(\gamma)\leq g(\gamma)$.  Finally, for statement five we observe that if $s\neq 0$, then there must be a vertex of non-zero genus, hence $\beta_1(\gamma)<g(\gamma)$, and so the total degree of $3g+n-3$ can not be realized in this case, from which the claim follows. \end{proof}

\subsection{Genus 1}

Let $\mathsf{P}_n$ be the modular graph corresponding to the polygon with $n$ vertices, each of genus $0$, and $n$ legs, one at each vertex and labeled in the dihedral order.  We view $\mathsf{P}_n$ as carrying an order on its edges induced by the dihedral order of the legs (so the edge connecting 1 to 2 comes first, followed by the edge connecting 2 to 3 etc.).  After Equation $\ref{mp}$, the associated Massey product $\mu_{\mathsf{P}_n}$ is an operation
\begin{equation}
H_\ast(\Gamma_{0,3})^{\tensor n}\stackrel{\mu_{\mathsf{p}_n}}\longrightarrow H_{\ast+n-1}(\Gamma_{1,n}).
\end{equation}

\begin{lemma}\label{genus1} For each $j\geq 1$, 
	$\mu_{\mathsf{P}_{2j+2}}=0$ and $\mu_{\mathsf{P}_{2j+1}}\neq 0$.
\end{lemma}
\begin{proof}
	Since $H_\ast(\Gamma_{0,3})=H_0(\Gamma_{0,3})$, the target of $\mu_{\mathsf{P}_n}$ is $H_{n-1}(\Gamma_{1,n})$ which by Lemma $\ref{Lieg1}$ is non-zero only if $n$ is odd and at least 3, whence the first statement.  
	
To prove the second statement, define a leaf vertex to be a vertex in a graph adjacent to at most one edge (recall edges are internal by convention).  Consider a modular graph $\gamma$ of type $(g,n)=(1,2j+1)$ with at least one edge and which has a leaf vertex of genus $0$.  By stability considerations, this leaf vertex must have at least two adjacent legs, labeled by $i$ and $j$ say.  Since the target of $\mu_\gamma$ is alternating, we must have $(ij)\mu_\gamma = - \mu_\gamma$. However, being genus $0$ the leaf vertex is labeled by the trivial representation, hence $\mu_\gamma = \mu_\gamma (ij)$.  Equivariance of the composition operations thus forces $\mu_\gamma=0$.
	
Next consider a modular graph $\gamma$ with $(g,n)=(1,2j+1)$ which has no leaf vertex of genus $0$.  Then $\gamma = \sigma\cdot \mathsf{P}_{2j+1}$ for some $\sigma \in S_{2j+1}$.  Indeed, if $\beta_1(\gamma)=0$, then $\gamma$ is a tree with at most one leaf vertex of genus not zero, and hence at least one leaf vertex of genus $0$.  And if $\beta_1(\gamma)=1$, then any edge which is not on the unique cycle of $\gamma$ could be removed to disconnect the graph into two components, one of which is a tree with genus $0$ vertices.  This tree is either a lone vertex (in which case this vertex was a leaf vertex in $\gamma$) or it contains at least two leaf vertices, at least one of which was a leaf vertex in the original $\gamma$.

Now suppose by way of contradiction that $\mu_{\mathsf{P}_{2j+1}}=0$. This would force $\mu_{\sigma\mathsf{P}_{2j+1}}=0$ for each permutation $\sigma$, and hence, by the preceding two paragraphs, it would force $\mu_\gamma=0$ for each graph $\gamma$ with $(g,n)=(1,2j+1)$.  This would in turn imply that the corolla, aka graph with 1 vertex, labeled by any non-zero element of $H_{2j}(\Gamma_{1,2j+1})^\ast$ was a cycle in $\FT(h_\ast(\Gamma))(1,2j+1)^{0,2j}$.  Having no edges, such a class couldn't be a boundary and therefore would represent a non-trivial class in homology.  Alas this contradicts Corollary $\ref{aacyclic}$, from which we conclude that $\mu_{\mathsf{P}_{2j+1}}\neq 0$. \end{proof}

Let us define $\alpha_{2j+1}:=\mu_{\mathsf{P}_{2j+1}}(1)$ for each $j\geq 1$.  In particular, the lemma tells us that each $\alpha_{2j+1}$ spans the alternating $S_{2j+1}$ representation $H_{2j}(\Gamma_{1,2j+1})$.  Define $\alpha^\ast_{2j+1}$ to be the unique linear functional sending $\alpha_{2j+1}$ to $1$.  Viewing $\alpha^\ast_{2j+1}$ as the label of a $2j+1$-corolla of genus $1$, we may view $\alpha_{2j+1}^\ast \in \FT(h_\ast(\Gamma))(1,2j+1)^{0,2j}$.  On the other hand, having specified an order on its edges, $\mathsf{P}_{2j+1}$ can be viewed as an element of $\FT(h_\ast(\Gamma))(1,2j+1)^{2j+1,0}$.  The proof of the Lemma immediately implies:

\begin{corollary}\label{dalpha} $d_{\FT}(\alpha^\ast) = \ds\sum_{\sigma\in S_{2j+1}} sgn(\sigma) \mathsf{P}_{2j+1}$.
\end{corollary}

\subsection{Collapsing an acyclic subcomplex} 
\begin{definition}  A vertex $v$ of an homogeneous $\gamma\in \FT(h_\ast(\Gamma))$ is called an egg if the degree of its vertex label at $v$ is $0$ and it either has genus $g(v)>0$ or is adjacent to a tadpole (or both). Define $Z(g,n)\subset \FT(h_\ast(\Gamma))(g,n)$ to be the span of all homogeneous elements containing at least one egg. 
\end{definition}

The terminology ``egg'' is chosen because the expansion differential creates (hatches) tadpoles.  Observe that for a non-zero $\gamma$, an egg can be adjacent to at most one tadpole, lest there be parallel edges adjacent to a degree $0$ vertex (an egg can hatch only one tadpole).

\begin{lemma}\label{acyclicZ} For each $(g,n)$, the graded vector space $Z(g,n)$ is an acyclic subcomplex of $\FT(h_\ast(\Gamma))(g,n)$.
\end{lemma}
\begin{proof}  We first argue that $Z(g,n)$ is a subcomplex.  If we apply the expansion differential at a vertex which is not an egg, the egg(s) of the original remain.   On the other hand if there is a unique egg and we apply the expansion differential at that vertex, then we can expand one edge only.  This is because the vertex is labeled by a class of degree $0$, and expanding $e$ edges lowers the degree by $e-1$.  If we expand (hatch) a tadpole, the vertex remains an egg.  If we expand a non-loop edge, then either the egg had genus $> 0$, in which case so does one of the vertices adjacent to the expanded edge, or the egg was adjacent to a tadpole to begin with.  In this case either the original tadpole remains, in which case the adjacent vertex is an egg, or the tadpole is split between the two vertices, in which case there are parallel edges adjacent to vertices labeled by classes of degree $0$, and the resulting differential term is $0$ after quotienting by the action of the automorphism group of the graph.  

To show the subcomplex $Z(g,n)$ has no cohomology we set up a suitable filtration.  Given a homogeneous $\gamma \in Z(g,n)$ define $t(\gamma)$ to be the number of tadpole edges of $\gamma$.  Define $z(\gamma)$ to be the number of vertices which are of degree $0$ and adjacent to a tadpole.  Note that each such vertex must then be adjacent to exactly one tadpole by parallel edge considerations, and so $z(\gamma)$ is also the number of tadpoles adjacent to a degree $0$ vertex.  We call these simple tadpoles.

We then filter $Z(g,n)$ by defining $Z(g,n)_r$ to be the span of all homogeneous elements $\gamma$ with $2|E(\gamma)| - t(\gamma) - z(\gamma) \geq r$.  This defines a bounded and exhaustive decreasing filtration
$$
Z(g,n)= Z(g,n)_0 \supset \dots \supset Z(g,n)_r \supset Z(g,n)_{r+1}\supset \dots \supset Z(g,n)_{6g+2n-6}\supset 0.
$$

Let us verify that this is indeed a filtration, i.e.\ that it is compatible with the expansion differential.  Let $\gamma^\prime$ be a non-zero homogeneous element appearing in the expansion differential of $\gamma$, formed by expanding a given vertex to a graph $\gamma^{\prime\prime}$ with $e$ edges, $v$ vertices and first Betti number $\beta$.  Furthermore let $x$ be the number of expanded edges which are not tadpoles. The tadpoles of $\gamma^\prime$ were either tadpoles of $\gamma$ or were expanded in the differential.  We thus conclude
$$t(\gamma^\prime) \leq e -x + t(\gamma).$$
Similarly, the simple tadpoles of $\gamma^\prime$ were either simple tadpoles of $\gamma$ or are adjacent to one of the vertices of $\gamma^{\prime\prime}$.  We thus conclude
$$z(\gamma^\prime) \leq z(\gamma) + v.$$
From these two inequalities we have:
\begin{eqnarray}
2|E(\gamma)| - t(\gamma) - z(\gamma)  \leq & 2|E(\gamma)|+ e+v-x-t(\gamma^\prime)  - z(\gamma^\prime) \label{ineq1} \\ \nonumber = & 2|E(\gamma^\prime)|-e+v-x -t(\gamma^\prime) - z(\gamma^\prime) 
 \\ \nonumber  = &   2|E(\gamma^\prime)|-\beta+1  -x-t(\gamma^\prime) - z(\gamma^\prime) 
 \\ \label{ineq2}  \leq & 2|E(\gamma^\prime)|-t(\gamma^\prime) - z(\gamma^\prime) \ \ \ \ \ \ \  \ \ \ \ \ \ \ \ \ 
\end{eqnarray}
For the last inequality note that $\beta\geq e-x$ and hence $\beta+x\geq 1$.  Hence the filtration degree never decreases upon applying the differential, as desired.

The differential fixes the filtration degree precisely when the two inequalities $\ref{ineq1}$ and $\ref{ineq2}$ are in fact equal.  The latter is an equality if and only if $x+\beta=1$.  Given $x+\beta=1$, the former is an equality if and only if $x=0$, $\beta=1$ and the degree of the expanded vertex is zero.  Indeed if $x=1$ and $\beta=0$ then $v=2$, but at most one new simple tadpole could be created (by separating a non-simple tadpole from a higher degree vertex), hence $z(\gamma^\prime)  \leq 1+z(\gamma)< v+z(\gamma)$ and so inequality $\ref{ineq1}$ is not an equality. On the other hand if $x=0$, $\beta=1$ and the degree of the expanded vertex were not zero then the expanded tadpole would not be simple and in this case  $z(\gamma^\prime)  < 1+z(\gamma) = v+z(\gamma)$.  Whence one implication.  For the converse, direct inspection shows that expanding a tadpole from a degree $0$ vertex results in equalities in both $\ref{ineq1}$ and $\ref{ineq2}$.

Now consider the associated graded $(Z(g,n)_r/Z(g,n)_{r+1}, d)$.  By the above analysis, 
a term in the differential corresponding to a given vertex is non-vanishing only if it expands a tadpole at an egg. By parallel edge considerations each degree $0$ vertex can support at most one tadpole, so the non-zero differential terms are precisely those which expand a tadpole at a degree $0$ egg which had no tadpoles to begin with.  Thus, the associated graded admits a splitting with two homogeneous elements belonging to the same summand if and only if their underlying graphs are related by a sequence of expansions or contractions of simple tadpoles.  Note that the data of which tadpoles are simple (or equivalently which vertices are eggs) is not specified solely by the underlying graph; it also requires specifying the degree of the vertex labels of the potential eggs (i.e.\ those vertices which have positive genus or are adjacent to a tadpole).  In particular, the splitting of $(Z(g,n)_r/Z(g,n)_{r+1}, d)$ is indexed by graphs along with a choice of a nonempty subset of these potential egg vertices.  After this choice, we may then assert that the graph has no simple tadpoles, contracting them by convention, to specify a unique summand of the splitting.


Fix such a summand.  It is indexed by a modular graph $\gamma$ of type $(g,n)$ along with a choice of $m>0$ eggs, having filtration degree $r$ and no simple tadpoles.  Define $Y$ to be the set of eggs of $\gamma$ and define $G=Aut(\gamma)$.  Note that $G$ acts on $Y$ since any automorphism of $\gamma$ must permute the set of eggs. 
 We claim that the summand of the associated graded corresponding to $\gamma$ is isomorphic to
\begin{equation}\label{ag}
	\overline{C}_\ast(\Delta_Y)\tensor_G h_\ast(\Gamma)^\ast(\gamma)\hookrightarrow Z(g,n)_r/Z(g,n)_{r+1},
\end{equation}
where $\overline{C}_\ast(\Delta_Y)$ denotes the augmented simplicial chain complex of the abstract simplicial complex consisting of non-empty subsets of $Y$.  In particular $\overline{C}_\ast(\Delta_Y)$ is isomorphic to the simplicial chains of an $m-1$ simplex.  Note that this claim will finish the proof, as the homology of this complex is (rationally) a quotient of zero, hence is zero.  This will then prove the associated graded, and hence $Z(g,n)$, has no homology.

Choose a total order on $Y$.  For each (possibly empty) subset $X\subset Y$, let $\gamma_X$ be the graph formed by hatching the eggs which are not in $X$ to become simple tadpoles.  There is a canonical isomorphism $h_\ast(\Gamma)^\ast(\gamma)\cong h_\ast(\Gamma)^\ast(\gamma_X)$, and we use this isomorphism to define a map
\begin{equation}\label{ag2}
	\overline{C}_\ast(\Delta_{Y})\tensor h_\ast(\Gamma)^\ast(\gamma)\stackrel{\pi}\to Z(g,n)_r/Z(g,n)_{r+1}
\end{equation}
sending $X\tensor h_\ast(\Gamma)^\ast(\gamma) \to h_\ast(\Gamma)^\ast(\gamma_X)_{Aut(\gamma_X)} \subset Z(g,n)_r/Z(g,n)_{r+1} $, where the simple tadpoles are ordered last among the edges of $\gamma_X$, and in the order prescribed by the chosen order on $Y$.

The statement that such a map would be dg is nearly immediate from the description above, since the differential applied to any term $X$ simply removes/expands the elements of $X$ one at a time in each case.  However, some care is needed in handling the signs. Let $x\in X\subset Y$, and suppose that $x$ is proceeded by $r$ elements of $X$ and followed by $s$ elements of $Y \setminus X$.  When an element $x$ is removed from $X$ in the simplicial chain complex, it is accompanied by the sign $(-1)^r$.  On the other hand, when the corresponding simple tadpole is expanded via the differential of the Feynman transform, this edge is placed in the last (rightmost) position by convention, and can be put back in order among $Y\cup \{x\}$ at the expense of a factor of $(-1)^s$.  To account for this, we define the map $\pi$ in equation $\ref{ag2}$ to include a factor of the sign of the shuffle permutation which moves the vertices of the given simplex to the left of the vertices not in the simplex. The product of the signs of the two shuffle permutations corresponding to the simplices $X$ and $X\setminus x$ is equal to $(-1)^{r+s}$, hence $\pi$ is dg.

The map defined in equation $\ref{ag2}$ surjects onto the $\gamma$-summand of the associated graded, so it remains to determine its kernel.  Every element in the source of $\pi$ can be written as a linear combination of elements of the form $X\tensor \{a_v\}$, where $\{a_v\}$ represents the labels of the vertices $v$.  The group $G$ acts on this source by $\sigma(X\tensor \{a_v\}) = \sigma(X)\tensor \sigma\{a_{v}\}$, where $\sigma$ acts by the  vertices and hence the vertex labels by permutation.  To finish the proof it suffices to verify that $ker(\pi)$ is equal to the vector space generated by all elements of the form $\sigma (X \tensor \{a_v\}) - X \tensor \{a_v\}$.

The fact that each difference of the form $\sigma (X \tensor \{a_v\}) - X \tensor \{a_v\}$ is in the kernel of $\pi$ follows immediately from the definition of the Feynman transform, since the terms in this difference map to elements related by an isomorphism in the target.


So it remains to show that the latter is contained in the former.  For this, first observe that if $\sum X_i \tensor \{a^i_v\} \in ker(\pi)$, then so is each sum over the subset of those terms whose graphs $\gamma_{X_i}$ are isomorphic, since non-isomorphic graphs map to different summands in the target.  So we may assume without loss of generality that $\gamma_{X_i}\cong \gamma_{X_j}$ for each pair of indices in the sum. Let $\sigma_i$ be an isomorphism between $\gamma_{X_i}$ and $\gamma_{X_1}$.  By the previous inclusion, the difference $X_i \tensor \{a^i_v\} - X_1 \tensor \sigma_i^{-1}\{a^i_v\}$ is in the kernel of $\pi$.  Summing up over all $i$, and subtracting $\sum X_i \tensor \{a^i_v\}$, the result is still in the kernel.  But this element is of the form $X_1\tensor \{b_v\}$, for some labels $\{b_v\}$, and such an expression can't be in the kernel without being $0$, since it simply expands a subset of tadpoles from the original labeled graph.  We thus conclude that $\sum X_i \tensor \{a^i_v\} = \sum (X_i \tensor \{a^i_v\} - X_1 \tensor \sigma_i^{-1}\{a^i_v\}) = \sum (X_i \tensor \{a^i_v\} - \sigma_i^{-1} (X_i \tensor \{a^i_v\})$, which verifies the other inclusion.
\end{proof}

\subsection{Model Theorem}
Observe that since $Z(g,n)$ is a bigraded and $S_n$-invariant subcomplex of the Feynman transform, the quotient $\FT(h_\ast(\Gamma))(g,n)/Z(g,n)$ inherits a bigrading and $S_n$ action as well.
\begin{definition} For a stable pair $(g,n)$, we define the chain complex
		\begin{equation*}
	\mathsf{L}(g,n) :=\left[(\FT(h_\ast(\Gamma))(g,n)/Z(g,n))\right]^{\ast,s>0},
	\end{equation*}
with differential induced by passage to the quotient by the $s=0$ subcomplex
		\begin{equation} \label{proj}
\left[(\FT(h_\ast(\Gamma))(g,n)/Z(g,n))\right] \twoheadrightarrow \mathsf{L}(g,n).
\end{equation}	
\end{definition}

The complex $\mathsf{L}(g,n)$ can thus be described informally as the span of stable, genus $g$ graphs with $n$ labeled legs, such that each vertex $v$ is labeled by a class in $H_\ast(\Gamma_{g(v),a^{-1}(v)})^\ast$, and subject to the following additional conditions regarding the degrees of vertex labels: at least one vertex has non-zero degree, no tadpole is adjacent to a vertex of degree $0$, and no higher genus vertex is labeled by a class of degree $0$.   Strictly speaking such a labeled graph represents an element only after specifying an order of the edges, up to even permutation.  The differential is given by expanding subgraphs at vertices of such graphs, via linear dualization of Massey products. 

We remark that coloring the vertices which have degree $0$ and genus $0$ black and the vertices which have degree non-zero and genus non-zero red recovers the description given in the introduction.  Note that by definition $\mathsf{L}(0,n)=0$ since in the case $g=0$ the complex $\FT(h_\ast(\Gamma))(0,n)$ is supported on internal degree $s=0$.  From now on we consider $\mathsf{L}(g,n)$ only when $g>0$.

\begin{theorem}  For each $g\geq 2$, the connecting homomorphism in the long exact sequence associated to the surjection in Equation $\ref{proj}$ induces an isomorphism
	\begin{equation*}
	H^i(\mathsf{L}(g,0)) \stackrel{\cong}\longrightarrow	H^{i-2g+1}(\mathsf{GC}_2^g),
	\end{equation*}	
and hence, after \cite{WTw}, an isomorphism 
\begin{equation*}
\mathfrak{grt}_1\cong \bigoplus_{g\geq 3} H^{2g-1}(\mathsf{L}(g,0)).
\end{equation*}
\end{theorem}	

\begin{proof}  Given $g\geq 2$, and setting $n=0$, we first identify $\Sigma^{2g}\mathsf{GC}_2^g$ with the kernel of Equation $\ref{proj}$.  Indeed, the $S_n$ action is vacuous in this case, and the kernel is the quotient of $\mathsf{FT}(\overline{\mathsf{Com}})(g,0)$ by the subcomplex of those graphs having an egg.  Since each such graph in the kernel is labeled only by degree $0$ elements, we quotient by those $(g,0)$-graphs containing tadpoles and/or higher genus vertices.  The result is tadpole free, stable, connected graphs of genus $g$ with odd edges and no legs, along with the expansion differential.  This is precisely $\mathsf{GC}_2^g$ up to a shift in degree.  To compare degrees, we recall that the degree of an homogeneous element in $\mathsf{GC}_2^g$ is $E-2g$, and in $\mathsf{FT}(\overline{\mathsf{Com}})(g,0)$ it is $E$.

We then use the fact that $\FT(h_\ast(\Gamma))(g,0)$ has no cohomology (Corollary $\ref{aacyclic}$) along with the fact that $Z(g,0)$ has no cohomology (Lemma $\ref{acyclicZ}$) to conclude their quotient has no cohomology as well.  Appealing to the long exact sequence associated to the surjection in Equation $\ref{proj}$ then proves the claim. 
\end{proof}

The theorem above refers to the special case that $n=0$. For general $n$, the kernel of Equation $\ref{proj}$ is a pointed version of the graph complex $\mathsf{GC}_2$, which we denote by $\Sigma^{2g}\mathsf{GC}^{g,n}_2$ in biarity $(g,n)$.  As above, the connecting homomorphism in the associated long exact sequence will be an isomorphism by Lemma $\ref{acyclicZ}$, provided $g>0$.  We may in turn identify the cohomology of $\mathsf{L}(g,n)$ with the cohomology of $\mathsf{FT(\overline{Com})}(g,n)$ by invoking \cite[Theorem 1.1(2)]{CGP2} which in our language states that the intersection $Z(g,n)\cap \mathsf{FT(\overline{Com})}(g,n)$ has no cohomology (for $g>0$).  In loc.cit.\ the authors interpreted $\mathsf{FT(\overline{Com})}$ as the cochains on  $\Delta_{g,n}$, the moduli space of tropical curves of constant volume, and so in this notation we have:

\begin{theorem}\label{gnthm}  
	 Let $g>0$ and $i\geq 0$. The connecting homomorphism in the long exact sequence associated to the surjection in Equation $\ref{proj}$ induces an isomorphism
	\begin{equation*}
	H^{i}(\mathsf{L}(g,n)) \stackrel{\cong}\longrightarrow H^{i+1}(\Delta_{g,n}).
	\end{equation*}	
\end{theorem}	
\begin{proof}  As above, we use the fact that $\FT(h_\ast(\Gamma))(g,n)$ has no cohomology (Corollary $\ref{aacyclic}$) along with the fact that $Z(g,n)$ has no cohomology (Lemma $\ref{acyclicZ}$) to conclude that the connecting homomorphism is an (degree 1) isomorphism from the cohomology of $\mathsf{L}(g,n)$ to the cohomology of $\Sigma^{2g}\mathsf{GC}_2^{g,n}$, which we in turn identify as the quotient of $\mathsf{FT(\overline{Com})}(g,n)$ by $Z(g,n)\cap \mathsf{FT(\overline{Com})}(g,n)$.  The claim then follows from the fact that this intersection has no cohomology \cite[Theorem 1.1(2)]{CGP2}.  We remark that a proof of this final fact can be seen as a version of our proof of Lemma $\ref{acyclicZ}$ above in which all tadpoles are simple, which then would allow us to use the simpler filtration edges minus tadpoles.	
\end{proof}


\section{Genus 2 versus depth 2.}

The goal of this section is to identify the cohomology of $\mathsf{L}(2,n)_{\Lambda_n}$ with the depth 2 modulo depth 3 summand of the associated graded of $\mathfrak{grt}_1$.  The first step is to identify a set of generators for the cohomology, which are given by even sided polygons with a distinguished genus $1$ vertex.

\subsection{Generators}

Let $r,s\geq 3$ both odd.  Define $P_{r,s}$ to be graph formed by attaching one leg to each vertex in an $r+1$-gon, except for a distinguished vertex, $v$, at which $s-2$ legs are attached (see Figure $\ref{fig:prs2}$).  Labeling $v$ to have genus $1$, and the remainder of the vertices as genus $0$, the graph $P_{r,s}$ specifies a $1$-dimensional subspace of the chain complex  $\mathsf{L}(2,r+s-2)_{\Lambda_n}$.  To fix a specific non-zero element in this subspace, we first choose an auxiliary total order on the union of the set of edges and legs as follows.  First choose any order on the legs at $v$, which gives us the first $s-2$ elements in the set.  Then, choose an edge adjacent to $v$ as the next element in the set.  Finally travel around the circuit of the $r+1$-gon ordering the remaining edges and legs at their point of first adjacent contact.  Specifically this means alternating edge, leg, edge, leg,... until arriving at the remaining edge adjacent to $v$, which is taken to be the maximal element.   The inherited order on the set of the $s$ flags adjacent to $v$ specifies a non-zero element of $H_{s-1}(\Gamma_{1,a^{-1}(v)})^\ast$, which we take as the label of $v$.   This convention determines a unique non-zero element, independent of the choices made, specifically the choice of leg order at $v$ and the choice of first/last adjacent edge at $v$.  This is because the polygon is odd and the label of $v$ is alternating.  We call this element 
\begin{equation}\label{nu}
\nu_{r,s} \in\mathsf{L}(2,r+s-2)_{\Lambda_n}^{r+1,s-1}.
\end{equation}
\begin{figure}
	\includegraphics[scale=.7]{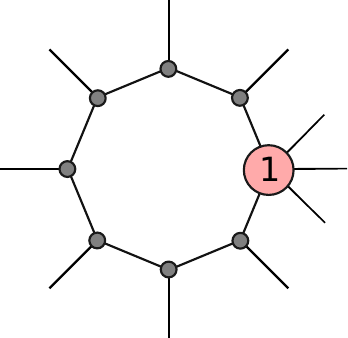}
	\caption{The graph $P_{7,5}$.}
	\label{fig:prs2}
\end{figure}

By Lemma $\ref{upper}$, $\nu_{r,s}$ is a cycle.  Specifically this is seen by applying the final case of the lemma to $\mathsf{ft}(h_\ast(\Gamma))(2,r+s-2)^{r+m+1,s-m}$.   Our next goal will be to compute the image of the cohomology class $[\nu_{r,s}]$ under the isomorphism in Theorem $\ref{gnthm}$.  For this we will first recall some results concerning the cohomology of $\mathsf{GC}_2^{2,n}$ from \cite{CCTW}.


\subsection{Theta complex}\label{thetasec}
The work of \cite{CCTW} shows that the anti-invariants of the complex $\mathsf{GC}_2^{2,n}$ are quasi-isomorphic to the associated complex of theta graphs, as we now recall.  Recall that a theta graph refers to a graph of type $(2,n)$ having two distinguished vertices connected by three sequences of edges, called strands, such that each non-distinguished vertex has one leg, and is thus trivalent.  We moreover consider our theta graphs to come with an order, well defined only up to even permutation, on its set of legs and edges, so as to specify a unique element in $(\mathsf{GC}_2^{2,n})_{\Lambda_{n}}$.  More precisely, we view a theta graph to be graded by its number of edges, and so is itself an element in $(\Sigma ^{2g}\mathsf{GC}_2^{2,n})_{\Lambda_{n}}$.

Let us use the following notation for theta graphs.  Given $0\leq a\leq b\leq c$ and $l\in \{0,1,2\}$ with $a+b+c+l=n$ we define $[a,b,c]^l\in (\mathsf{GC}_{2}^{2,n})_{\Lambda_n}$ to be the theta graph with strands having $a$, $b$ and $c$ legs respectively, $l$ of whose distinguished vertices also have a leg.  We will fix a specific element for each $a,b,c,l$ by specifying the order of the edges and legs to be as follows.  First we order the $l$ legs labeling the distinguished vertices.  Then we order the legs and edges along the $a$-strand, then the $b$-strand, then the $c$-strand, starting from the distinguished vertex adjacent to the (lesser) leg when $l=1$ (resp.\ $l=2$).  When $l$ is even, note that a reflection which interchanges the distinguished vertices produces a sign of  $(-1)^{a+b+c+\frac{l}{2}}$, and so the element vanishes if $n$ is odd and $l=0$ and also vanishes if $n$ is even and $l=2$.  Hence, our convention specifies a single element in the coinvariants of the automorphism group action, regardless of which side of the theta we started on in the case $l=0$.  Also note that if $a=b$ or $b=c$, the automorphism which interchanges the strands is multiplication by $-1$, and so the element vanishes.  We thus assume $0\leq a<b<c$ from now on.  The degree of  $[a,b,c]^l$ is $a+b+c+3$.

We remark here that our sign conventions differ slightly from those of \cite{CCTW} as well as \cite{MFelder}.  Given our conventions, a direct calculation shows that the expansion differential is:
\begin{equation}\label{d1}
d([a,b,c]^1) =  [a+1,b,c]^0 - [a,b+1,c]^0+ [a,b,c+1]^0,
\end{equation}
and if $a+b+c$ is odd:
\begin{equation}
d([a,b,c]^2) =  2[a+1,b,c]^1-2[a,b+1,c]^1+2[a,b,c+1]^1.
\end{equation}

Define $\Theta_{n}\subset (\Sigma^{2g}\mathsf{GC}_{2}^{2,n})_{\Lambda_n}$ to be the subcomplex spanned by all such theta graphs, i.e.\ all the elements $[a,b,c]^l$ such that $a+b+c+l=n$.  That $\Theta_n$ is a subcomplex (under the expansion differential) is immediate.  By the above analysis, when $n$ is even, there is an isomorphism $\Theta_n\stackrel{\cong}\rightarrow \Theta_{n+1}$ given by ``integration'', i.e.\ $[a,b,c]^l \mapsto \frac{1}{l+1}[a,b,c]^{l+1}$.

The inclusion $\Theta_{n}\subset (\Sigma^{2g}\mathsf{GC}_{2}^{2,n})_{\Lambda_n}$ is a quasi-isomorphism \cite[Theorem 3.1]{CCTW}.  Strictly speaking, \cite{CCTW} considers a slightly larger commutative graph complex which allows tadpoles, but it has the same cohomology in genus $2$, after \cite[Theorem 1.1(1)]{CGP2}.  Combining this with Theorem $\ref{gnthm}$ we have isomorphisms
\begin{equation}\label{isos}
H^{\ast+2}(\mathsf{L}(2,n)_{\Lambda_n}) \stackrel{\cong}\to H^{\ast-1}(\mathsf{GC}^{2,n}_2)_{\Lambda_n}\stackrel{\cong}\leftarrow H^{\ast+3}(\Theta_n).
\end{equation}

Below we restrict attention to even $n$.  For example when $n=10$ only the following are non-zero,
\begin{equation}
H^{12}(\mathsf{L}(2,10)_{\Lambda_{10}}) \stackrel{\cong}\to H^9(\mathsf{GC}^{2,10}_2)_{\Lambda_{10}}\stackrel{\cong}\leftarrow H^{13}(\Theta_{10})
\end{equation}
generated on the right hand side by the classes of $[0,2,8]^0$ or $[0,4,6]^0$, which are subject to one linear relation.

\subsection{Cohomology generators of $\mathsf{L}(2,n)_{\Lambda_n}$}

We next calculate the preimage of the generators $[0,r-1,s-1]^0$ under the isomorphism induced by Equation $\ref{isos}$ and identify them, up to non-zero scalar multiples, as the classes of $\nu_{r,s}$.

\begin{proposition}\label{chain}  Let $r>s$ both odd and $n=r+s-2$.  The isomorphism $H^{n+2}(\mathsf{L}(2,n)_{\Lambda_n})\to H^{n+3}(\Theta_n)$ induced via Equation
$\ref{isos}$ sends $\ds\frac{-[\nu_{r,s}]}{(s-2)!}$ to the class of $[0,r-1,s-1]^0$.	
\end{proposition}
\begin{proof}  The isomorphism in Theorem $\ref{gnthm}$ can be determined by viewing $\nu_{r,s}\in \FT(h_\ast(\Gamma))(2,n)_{\Lambda_n}$ and computing its differential.  The result will represent the image of $[\nu_{r,s}]$ under the connecting isomorphism.  The terms of $d(\nu_{r,s})$ are given by expanding the vertex $v$ to a stable $s$-gon, two of whose legs are joined to the original $r+1$-gon.  Therefore the differential must have the form:
\begin{equation*}
d(\nu_{r,s}) = \sum_{q=0}^{\frac{s-3}{2}} c_q[q,s-2-q,r]^0,
\end{equation*}
for some coefficients $c_q$.  Note that every differential term is of the form $[-,-,r]^0$ since $r>s$.

On the other hand a direct calculation shows
\begin{equation}\label{dcalc}
d(\ds\sum_{q=0}^{\frac{s-3}{2}} [q,s-2-q,r-1]^1) = -[0,s-1,r-1]^0 + \sum_{q=0}^{\frac{s-3}{2}} [q,s-2-q, r]^0.
\end{equation}
This calculation is done by invoking Equation $\ref{d1}$, making successive cancellations until arriving at a term of the form $[\frac{s-1}{2},\frac{s-1}{2},r-1]^0$, which vanishes.  Thus, to prove the claim it remains to show that the coefficient $c_q$ above is independent of $q$, and equals $-(s-2)!$.

The differential $d$ which is applied to $\nu_{r,s}$ is the passage to coinvariants of the (weak) Feynman transform differential.  To describe it explicitly we first choose a representative of the coinvariant class $\nu_{r,s}$ by labeling its legs by the set $\{1\cdc n\}$.  By convention, we choose the order which matches the order given in the definition of $\nu_{r,s}$ -- namely first label the legs of the genus $1$ vertex, then travel around the $r+1$-circuit labeling successively at the point of first contact.  Call this labeled element $\nu_{r,s}^n \in \FT(h_\ast(\Gamma))(2,n)$.

The only vertex of $\nu_{r,s}^n$ which permits an expansion is the genus $1$ vertex.  So to apply the Feynman transform differential to $\nu_{r,s}^n$ we first consider the Feynman transform differential applied to the $s$-corolla with label $\alpha_s^\ast\in H_{s-1}(\Gamma_{1,s})^\ast$.   By Corollary $\ref{dalpha}$, the terms in the differential are given by expanding this vertex to an $s$-gon with all possible leg labels.   By convention, above, the induced edge order appearing in such a differential term is the class of any standard dihedral order, all of which are related by an even permutation.  There are $(s-1)!/2$ such terms in this differential.

To each term in $d(\alpha^\ast_s)$ above, there is an associated differential term in  $d_{\FT}(\nu_{r,s}^n)$ given by gluing in the $r$-strand to the legs labeled $s-1$ and $s$, with orientation (by convention) chosen so that it runs from $s$ to $s-1$.  These $\frac{(s-1)!}{2}$ terms constitute the differential $d_{\FT}(\nu_{r,s}^n)$ up to an overall sign.  This overall sign is given by the Leibniz rule which moves the differential past the modular operadic compositions, and since $r+1$ and $s-1$ are both even, the only contribution comes from moving $d$ over the final self gluing of the $r$-strand, and hence is $-1$.

Each term in the differential specifies a leg-labeled theta graph whose long strand has $r$ legs.  Since the legs are labeled, the number of terms whose shortest strand has $q$ legs is independent of $q$ and is equal to $(s-2)!$.  Passing to anti-coinvariants (using the leg order induced from the labeling) we see that each of the differential terms is a theta graph of type $[q,s-2-q,r]^0$ for some $q<s-2-q$. 

Let us now show that two theta graphs of type $[q,s-2-q,r]^0$ appearing in $d(\nu_{r,s})$ appear with the same sign after passage to anti-coinvariants.  Two such terms arise after gluing in the $r$-strand to corresponding labelings of the $s$-gon.  These two labeled $s$-gons correspond to terms of $d(\alpha^\ast_s)$, and so are related by a permutation of the labels $\{1\cdc s-2\}$ (since they have the same shortest strand, we can assume this permutation fixes $s$ and $s-1$).  If $\tau$ is this permutation then the coefficients of the corresponding terms in $d(\alpha^\ast_s)$ differ by a factor of $sgn(\tau)$ (note the edge order plays no role here because the two edge orders must be related by an even permutation).  On the other hand, comparing these terms in $d_{\FT}(\nu_{r,s}^n)$ would require applying the same permutation after gluing in the $r$-strand, but before passing to coinvariants.  After passage to coinvariants this too would produce $sgn(\tau)$.  Thus in $d(\nu_{r,s})$ these two terms appear with the same sign.

From this we conclude
\begin{equation*}
d(\nu_{r,s}) = -(s-2)!\sum_{q=0}^{\frac{s-3}{2}} \pm[q,s-2-q,r]^0
\end{equation*}
and it remains to check that the signs inside this sum coincide for each $q$.

Observe that for any $q$, we can produce a theta graph of type $[q,s-2-q,r]^0$ by first applying a transposition of $s-1$ with $q+1$ to the leading term (standard dihedral order) of $d(\alpha^\ast_s)$ and then gluing in the $r$-strand. In particular, gluing in the $r$-strand before transposing gives a theta graph of type $[0,s-2,r]^0$, while gluing in the $r$-strand after transposing gives a theta graph of type $[q,s-2-q,r]^0$.  It remains to compare the signs with which these terms appear in $d(\nu_{r,s})$. 

Recall our convention for signs of theta graphs is that $[a,b,c]^0$ takes $a$ then $b$ then $c$ from ``left to right'' on each strand (the result being independent of the planar embedding).  To determine the sign of the $[0,s-2,r]^0$ term note while its strands are ordered from left to right, the shortest strand (with no flags) is second, so has to be moved over the $s-1$ edges of the middle strand, which produces a plus sign since $s$ is odd.  So we get $+[0,s-2,r]^0$ appearing.

To determine the sign of $[q,s-2-q,r]^0$ note that the short strand has $1$ through $q$ starting at the vertex formally adjacent to $s$ as desired, and the third strand is also correct by convention.  To make the comparison to our standard convention it suffices to correct the second strand (since it's oriented from right to left after the gluing).  This requires ``reflecting'' both the legs and edges of strand.  For the legs this requires $\floor{\frac{s-2-q}{2}}$ transpositions and for the edges it requires $\floor{\frac{s-q}{2}}$ transpositions so the sign is $-1$, regardless of $q$.  But this term also carries a minus sign in the sum $d(\alpha_s^\ast)$ (because it came from a transposition applied to the leading term) so the total result is a plus sign, and we also get $+[q,s-q-2,r]$ appearing in $d(\nu_{r,s})$.

We've thus shown, with these conventions, that 
\begin{equation*}
d(\nu_{r,s}) = -(s-2)!\sum_{q=0}^{\frac{s-3}{2}} [q,s-2-q,r]^0
\end{equation*}
and so Equation $\ref{dcalc}$ implies that $d(\nu_{r,s})$ and $-(s-2)![0,s-1,r-1]^0$ represent the same class in the cohomology of $\mathsf{GC}_{2}^{2,n}$, from which the claim follows.
\end{proof}

\subsection{Isomorphism Theorem}	
For $3\leq s <r$ both odd define $\rho_{r,s} = \ds\frac{-\nu_{r,s}}{(s-2)!}$.  The calculation in Proposition $\ref{chain}$ allows us to prove:

\begin{theorem}\label{g2thm}	The unique linear map which sends $[\rho_{r,s}]\mapsto \{\sigma_{r},\sigma_{s}\}$ for all $3\leq s <r$ both odd, induces an isomorphism 
	\begin{equation*}
	\bigoplus_{n\geq 4, \text{ even}} H^{\ast}(\mathsf{L}(2,n)_{\Lambda_n}) \stackrel{\cong}\longrightarrow \op{F}_2(
	\mathfrak{grt}_1)/\op{F}_3(
	\mathfrak{grt}_1).
	\end{equation*}
\end{theorem}
\begin{proof}  The difficult work has been done above.  Fix such an even $n$ and consider the following sequence of isomorphisms:
		\begin{equation}\label{maps}
	H^\ast(\mathsf{L}(2,n)_{\Lambda_n}) \stackrel{\cong}\to H^\ast(\Theta_{n})\stackrel{\cong}\rightarrow H^\ast(\Theta_{n+1})\stackrel{\cong}\rightarrow (\op{F}_2(\mathfrak{grt}_1)/\op{F}_3(\mathfrak{grt}_1))_{n+2}
	\end{equation}
The left most map is that from Proposition $\ref{chain}$.  The middle map is induced by chain level isomorphism described above (subsection $\ref{thetasec}$) and the rightmost map is the isomorphism sending the class of $[0,s-1,r-1]^1$ to $\{\sigma_r,\sigma_s\}$ as in \cite{MFelder}.  From the calculation in Proposition $\ref{chain}$ we see that this composite sends $[\rho_{r,s}]$ to $\{\sigma_{r},\sigma_{s}\}$, from which the claim follows.
\end{proof}
	
	By way of conclusion, let us observe that a choice of Massey products as made above determines an isomorphism between $H_\ast(\Gamma_{2,n})^\ast$ and the linear relations among the brackets in $(\op{F}_2(\mathfrak{grt}_1)/\op{F}_3(\mathfrak{grt}_1))_{n+2}$, for $n$ even. First, consider the short exact sequence.
	
	\begin{equation}\label{ses}
	0\to \widehat{\mathsf{L}}(2,n) \hookrightarrow \mathsf{L}(2,n) \twoheadrightarrow H_\ast(\Gamma_{2,n})^\ast_{\Lambda_n} \to 0
	\end{equation}
	where the right hand map is the projection to corollas and $\widehat{\mathsf{L}}$ denotes the kernel of this projection.  
	
	Combining \cite[Theorem 6.2]{CCTW} and Theorem $\ref{gnthm}$ above we know that $\mathsf{L}(2,n)$ has cohomology only in degree $n+2$. 
	 	By \cite[Theorem 2.10]{CHKV} we know that $H_\ast(\Gamma_{2,n})^\ast_{\Lambda_n}$ is supported in degree $n+1$.  Thus the associated long exact sequence is a short exact sequence of the form:
	\begin{equation*}
	\dots \to 0\to H^{n+1}(\Gamma_{2,n})^\ast_{\Lambda_n} \stackrel{\partial}\hookrightarrow H^{n+2}(\widehat{\mathsf{L}}(2,n))\stackrel{\pi}\twoheadrightarrow   H^{n+2}(\mathsf{L}(2,n))\to 0 \to \dots
	\end{equation*}
where $\partial$ denotes the connecting homomorphism.  Therefore the classes $[\nu_{r,s}]$ with $r>s$ are a basis for the homology of $\widehat{\mathsf{L}}(2,n)$ (they have the right dimension and are linearly independent since they belong to different summands).  We may thus identify 
$$H^{\ast}(\mathsf{L}(2,n)_{\Lambda_n}) = span\{ \nu_{r,s}\} / im(\partial),$$
and since $\partial$ is injective, this specifies an isomorphism between $H_\ast(\Gamma_{2,n})^\ast$ and the linear relations among the brackets in $(\op{F}_2(\mathfrak{grt}_1)/\op{F}_3(\mathfrak{grt}_1))_{n+2}$.

\end{document}